\documentclass[12pt]{amsart}
\usepackage{amsmath,amsthm,amsfonts,todonotes}
\newtheorem{thm}{Theorem}[section]
\newtheorem{cor}[thm]{Corollary}
\newtheorem{lem}[thm]{Lemma}
\newtheorem{prop}[thm]{Proposition}

\theoremstyle{definition}

\newcommand{\setdef}{\stackrel {\rm {def}}{=}}

\def \N{{\mathbb N}}
\def \Q{{\mathbb Q}}
\def \T{{\mathbb T}}

\def \R{{\mathbb R}}
\def \Log{{\rm{Log}}}

\newcommand{\Cal}{\mathcal}
\def\e{{\rm e}}
\def\eps{{\varepsilon}}

\newcommand{\tend}[3][]{\xrightarrow[#2\to#3]{#1}}

\begin{document}

\title{Mean ergodic theorems in $L^r(\mu)$ and $H^r(\mathbb T)$, $0<r<1$}

\author{\MakeLowercase{el} Houcein \MakeLowercase{el} Abdalaoui}
\address{Laboratoire Math\'ematique Rapha\"el Salem, Universit\'e de Rouen-Normandie, Rouen, France}
\email{elhoucein.elabdalaoui@univ-rouen.fr}

\author{Michael Lin}
\address{Department of Mathematics, Ben-Gurion University of the Negev, Beer-Sheva, Israel}
\email{lin@math.bgu.ac.il}

\subjclass[2010]{Primary:   37A30; secondary : 42B30, 37E10}
\keywords{ergodic theorem, measure preserving transformations, non-integrable functions, $L^r(\mu)$ with $0<r<1$, circle rotations, Hardy spaces  $H^r(\mathbb T)$}

\begin{abstract}
Let $T$ be the Koopman operator of a measure preserving transformation $\theta$ of a probability 
space $(X,\Sigma,\mu)$. We study the convergence properties of the averages 
$M_nf:=\frac1n\sum_{k=0}^{n-1}T^kf$ when $f \in L^r(\mu)$,  $0<r<1$. We prove that if 
$\int |M_nf|^r d\mu \to 0$, then $f \in \overline{(I-T)L^r}$, and show that the converse 
fails whenever $\theta$ is ergodic aperiodic. When $\theta$ is invertible ergodic aperiodic,
 we show that for $0<r<1$ there exists $f_r \in (I-T)L^r$ for which $M_nf_r$ does not converge a.e.  (although $\int |M_nf|^r d\mu \to 0$). We further establish that for $1 \leq p <\frac{1}{r},$ there is a dense $G_\delta$ subset
${\Cal F}\subset  L^p(X,\mu)$ such that $\limsup_n \frac{|T^nh|}{n^r}=\infty$ a.e.  
for any $h \in {\Cal F}$.

When $T$ is induced by an irrational rotation of $\mathbb T$, the Hardy spaces $H^r(\mathbb T)$ 
are $T$-invariant. For $0<r<1$, we prove that 
$H^r(\mathbb T)=\{constants\}\oplus \overline{(I-T)H^r(\mathbb T)}$, and 
$\int |M_nf|^r d\mu \to 0$ for $f \in \overline{(I-T)H^r(\mathbb T)}$. However, there exists 
$f \in (I-T)H^r(\mathbb T)$ such that $M_nf$ does not converge a.e.
\end{abstract}

\maketitle

\section{Introduction}
	
Let $\theta$ be a measure preserving transformation of a probability space $(X,\Sigma,\mu)$
and for $f$ measurable put $Tf(x)=f(\theta x)$. For $1 \le p< \infty$, the linear operator $T$ 
is an isometry of $L^p(\mu)$,  and the mean ergodic theorem yields the $L^p$-convergence of 
the averages $M_n f:=\frac1n\sum_{k=0}^{n-1}T^kf$.
\smallskip

When $0<r<1$, the space $L^r(\mu):=\{f: \int|f|^rd\mu<\infty\}$ is a vector space, which is a complete
metric space under the  (invariant) metric $d(f,g)=\int|f-g|^r d\mu$ \cite[p. 37]{Ru}, and $T$ is an 
isometry of this metric space. Note that an isometry in a complete metric space is continuous.
 Dalibor Voln\'y asked one of us about the ergodic properties of $T$ in $L^r$.

When $T$ is induced by an irrational rotation of $\mathbb T$, the Hardy spaces $H^r(\mathbb T)$ 
are $T$-invariant subspaces of $L^r(\mathbb T)$. For $0<r<1$, we study the ergodic properties of
$T$ in $H^r(\mathbb T)$.

	\section{Ergodic theorems in $L^r(\mu)$, $0<r<1$}

	In this section $0<r<1$ is fixed, and $f_n \to f$ for $\{f_n\} \subset L^r$ means convergence
	in the metric, i.e. $\int|f_n-f|^r d\mu \to 0$.
	
\begin{lem} \label{limit}
	Let $T$ be induced on $L^r(\mu)$ by a measure preserving transformation $\theta$, 
	and let $f \in L^r$.  If $\frac1n\sum_{k=0}^{n-1}T^k f\to g$, then $Tg=g$. Hence 
	$g$ is constant when $\theta$ is ergodic.
\end{lem}
\begin{proof}
Clearly $d(Tg,M_nf) \le d(Tg,TM_nf)+ d(TM_nf,M_nf)$. Since $T$ is an isometry, \newline
${d(Tg,TM_nf)=d(g,M_nf) \to 0}$. We also have
$$
d(TM_nf,M_nf)=\int |TM_nf-M_nf|^rd\mu=\int\Big|\frac{T^nf}n-\frac{f}n\Big|^rd\mu \le
2\int\Big|\frac{f}n\Big|^rd\mu \to 0,
$$
so $M_nf \to Tg$, hence $Tg=g$.
\end{proof}

\begin{cor}
Let $0<r<1$ and $T$ as above. If $f \in L^r$ satisfies $\frac1n\sum_{k=0}^{n-1}T^k f\to g$,
then $\frac1n\sum_{k=0}^{n-1}T^k (f-g) \to 0$.
\end{cor}

The next theorem is a partial analogue of the result in Banach spaces proved in \cite{Yo}. Yosida
used the Hahn-Banach separation theorem, but in $L^r$ there are no non-zero continuous
 functionals \cite[p. 37]{Ru}, so a different proof is needed.

\begin{thm}\label{main1}
Let $T$ be induced on $L^r(\mu)$ by a measure preserving transformation, and let 
$f \in L^r$.  If $f \in(I-T)L^r$, then $\frac1n\sum_{k=0}^{n-1}T^k f \to 0$. If
	$\frac1n\sum_{k=0}^{n-1}T^k f \to 0$, then $f \in \overline{(I-T)L^r}$.
\end{thm}
\begin{proof}
Let $f=h-Th$. Then, as before,
$$
\int |M_n f|^rd\mu= \int\Big|M_nh-TM_nh\Big|^rd\mu=
\int\Big|\frac{f}n -\frac{T^nh}n\Big|^rd\mu \to 0.$$ 
\medskip

Assume now that $M_nf \to 0$. Put 
$$
g_n:=f-M_nf=\frac1n\sum_{k=0}^{n-1} (f-T^kf)=\frac1n\sum_{k=1}^{n-1}(I-T)\sum_{j=0}^{k-1}T^jf.
$$
 Hence $g_n \in (I-T)L^r$, and $d(f,g_n)=\int|M_nf|^r d\mu\to 0$, so $f \in \overline{(I-T)L^r}$.
\end{proof}

We will show below that there is no complete analogue of Yosida's result.

\begin{thm} \label{sato}
Let $T$ be induced on $L^r$, $0<r<1$, by an ergodic probability preserving transformation $\theta$,
 and let $f \in L^r$. Then $f \in (I-T)L^r$ if and only if 
\begin{equation} \label{GH}
\sup_{n\ge 1} \int \big|\sum_{k=0}^{n-1}T^kf\big|^r d\mu< \infty.
\end{equation}
\end{thm}
\begin{proof}
If $f=(I-T)g$ with $g \in L^r$, then $\sum_{k=0}^{n-1}T^kf =g-T^ng$, so
$$
\sup_{n\ge 1}\int \Big|\sum_{k=0}^{n-1}T^kf\Big|^r d\mu \le 2 \int|g|^rd\mu < \infty.
$$

Assume now that \eqref{GH} holds. Then 
$$
\sup_{N\ge 1} \frac1N\sum_{n=1}^N\int \Big|\sum_{k=0}^{n-1}T^kf\Big|^r d\mu < \infty.
$$
We can therefore apply Sato's \cite[Corollary, p. 287]{Sato} and obtain that $f \in (I-T)L^r$.
\end{proof}
{\bf Remark.} Theorem \ref{sato} is valid for $1 <r< \infty$ as a special case of Browder \cite{Br}, 
and for $r=1$ by Lin and Sine \cite{LS}. One of the novelties in \cite{AO} and \cite{Sato} is in
 the results for  $0<r<1$.

\medskip

When $\theta$ is ergodic and $f \in L^1(\mu)$, the ergodic theorems yield that $M_nf \to \int fd\mu$
in $L^1$-norm and a.e., so the limit does not depend on the ergodic transformation.  When $f \in L^r$, $0<r<1$,
and $M_nf$ converges in the $L^r$ metric, the limit is still a constant, by Lemma \ref{limit},
but we show that  the constant may depend on $\theta$.
\medskip

\begin{prop} \label{major}
There exist two ergodic  measure preserving transformations on the unit circle, with induced 
operators $T$ and $S$, and a function $f \in L^r$, $0 <r<\frac13$, such that, in the $L^r$ metric,
$\frac1n\sum_{k=0}^{n-1}T^k f \to \mathbf{1}$ and $\frac1n\sum_{k=0}^{n-1}S^kf \to \mathbf{0}$,
\end{prop}
\begin{proof}
The two transformations $T$ and $S$  are given by irrational rotations $\alpha,\beta$ which are independent
over the rationals. Buchzolich \cite{Bu} constructed a measurable function $f$ (necessarily not
integrable) such that $\frac1n\sum_{k=0}^{n-1}T^k f \to 1$ and $\frac1n\sum_{k=0}^{n-1}S^kf \to 0$ a.e.
\smallskip

{\it Claim: It is possible to choose the parameters in \cite{Bu} so that $f \in L^r$, for $0<r< 1/3$.}

\noindent
Let $\eps_0\in (\frac12,1)$, and for $j \ge 1$ define inductively
$$
\frac1{\eps_j}=\Big[\frac1{\sqrt{\prod_{k=0}^{j-1}\eps_k}}\Big] +1.
$$
Then $\eps_1=\frac12 <\eps_0$. For any $j$ we have $\eps_j<1$ and $1/\eps_j$ an integer. 
By definition $\eps_j^2 <\prod_{k=0}^{j-1} \eps_k$. Since $\eps_j < 1$, $\ \eps_{j+1} < \eps_j < \eps_0$ 
for $j \ge 1$.  By induction, 
$\prod_{k=0}^j \eps_k \le  \eps_0^{j+1}$, hence $\eps_j^2 < \prod_{k=0}^{j-1}\eps_k \le \eps_0^j$.
Thus, for $0<r< 1/3$, we have
\begin{equation} \label{Lr}
\sum_{j \geq 0}\eps_{j}^{1-3r} \le \sum_{j=0}^\infty \eps_0^{j(1-3r)/2}= 
\sum_{j=0}^\infty (\eps_0^{(1-3r)/2})^j <+\infty.
\end{equation}
The sequence $\eps_j$ satisfies all the assumptions in \cite[p. 250]{Bu}.

The function $f$, constructed in \cite{Bu}, is given by $f=\sum_{j=0}^{+\infty}(-1)^{j}f_j$,
with the definition of $f_j$ dependent on the fixed irrationals $\alpha$ and $\beta$; the support of 
$f_j$, denoted by $E^{j}$, satisfies $\mu(E^j)<\eps_{j}.$ Furthermore, the construction yields

(i) $\int f_jd\mu=1$ for $j \ge 0$.

(ii)  $\displaystyle \sup_{x \in X} |f_{2j}(x)| \leq \eps_{2j}^{-1} \sup_{x \in X} |f_{2j-1}(x)|,$

(iii) $\displaystyle \sup_{x \in X} |f_{2j+1}(x)| \leq \eps_{2j+1}^{-1} \sup_{x \in X} |f_{2j}(x)|.$	

\noindent
By the construction $|f_0| \le 1$, so with $f_{-1} \equiv 1$ we have
 \begin{equation} \label{suprema}
\sup_{x \in X} |f_k(x)| \leq \eps_k^{-1} \sup_{x \in X} |f_{k-1}(x)|= 1/\prod_{j=0}^k \eps_j
\quad \rm{ for } \quad k \ge 0.
\end{equation}
Applying  the triangle inequality in $L^r$ and the estimate $\eps_j^2 < \prod_{k=0}^{j-1}\eps_k$,  we get
$$
\int_{X} |f|^r d\mu \leq \sum_{j \geq 0} \int |f_j|^r d\mu \leq \sum_{j=0}^\infty \mu(E^j) \sup_{x \in X}|f_j(x)|^r 
 \leq \sum_{j=0}^\infty  \eps_{j} \Big(\frac{1}{\prod_{k=0}^{j} \eps_{k}}\Big)^r \leq \sum_{j=0}^\infty  \eps_j^{1-3r}.
$$
Therefore, by \eqref{Lr}, $f \in L^r$ when $0<r<\frac13$, which proves the claim.
\smallskip

We now prove the claimed  convergence in $L^r$ metric, for fixed $0<r<\frac13$.  
We denote $M_n(T):=\frac1n\sum_{k=0}^{n-1}T^k$ and $M_n(S):=\frac1n\sum_{k=0}^{n-1}S^k$. We use the choice of $\eps_j$ made above, and the corresponding $f$, which is in $L^r$ 
by the claim.

We first  prove that $M_n(T)f \to\bf 1$ in $L^r$. Let $\eps>0$ and fix $J_0$ such that 
$\sum_{j=2J_0}^\infty \eps_{2j+1}^{1-3r} <\eps/2$.

In the construction in \cite{Bu}, it is shown that there exist sets $\hat X_{2j+1}$ with 
$\mu(\hat X_j)> 1-2\eps_{2j+1}$,  such that $|M_n(T)(f_{2j+2}-f_{2j+1})(x)| < \eps_{2j+1}$ 
for every $x \in \hat X_{2j+1}$, $n>0$ and $j\ge 0$. Put $\hat X=\cap_{j \ge 2J_0}\hat X_{2j+1}$.
Then $\mu(\hat X) > 1-\eps/2$.
Put $g_{J_0}:=f_0+\sum_{j=0}^{2J_0-1} (f_{2j}-f_{2j+1})$. Recall that
$f=\sum_{j=0}^\infty (-1)^jf_j=\sum_{j=0}^\infty (f_{2j}-f_{2j+1}$. Since 
$$
|M_N(T)f(x)-1| \le |M_N(T)g_{J_0}(x)-1| +\sum_{j=2J_0}^\infty |M_N(T)(f_{2j+2}-f_{2j+1})(x)|,
$$
the estimates for $x\in \hat X$ yield that on $\hat X$
$$
 |M_N(T)f-1| \le
|M_N(T)g_{J_0}-1| +\sum_{j=2J_0}^\infty |M_N(T)(f_{2j+2}-f_{2j+1})| $$
$$
\le   |M_N(T)g_{J_0}-1| + \sum_{j=2J_0}^\infty \eps_{2j+1} 
\le   |M_N(T)g_{J_0}-1| + \eps/2.
$$
By the ergodic theorem $|M_N(T)g_{J_0}-1| \to 0$ almost surely, since by the construction in
 \cite{Bu} $g_{J_0}$ is bounded with integral 1. Hence for $x \in\hat X$ and  $N\ge N_1(x)$
we have on $\hat X$ 
$$ |M_N(T)f(x)-1|< \eps.$$
Whence $|M_N(T)f-1|$ converge almost surely to $0$ (as shown in \cite{Bu}).
By the triangle inequality in $L^r$,
\begin{equation} \label{norms}
\int |M_N(T)f-1|^r d\mu \le
\int |M_N(T)g_{J_0}-1|^r d\mu + \sum_{j=2J_0}^\infty \int |M_N(T)(f_{2j+2}-f_{2j+1})|^r d\mu.
\end{equation}
Since $g_{J_0}$ is bounded, by the dominated convergence theorem 
$\int |M_N(T)g_{J_0}-1|^r d\mu \to 0$ as $N \to\infty$. We estimate the series in \eqref{norms}.
For each $j$, let $E^{2j+2,2j+1}$ be the support of $f_{2j+2}-f_{2j+1}$. It is shown in 
\cite[p. 251]{Bu} that for some sequence $(N_{j})$, we have

(i) $\mu(\bigcup_{0}^{N_{2j+2}}(T^{-k}E^{2j+2,2j+1}))<2\eps_{2j+2}.$

(ii) $\big|\sum_{k=0}^{M}(f_{2j+2}-f_{2j+1})(T^kx)\big|<2 \eps_{2j+2}^{-1}\sup_{y \in X}|f_{2j+1}(y)|$, 
for all $M \in \mathbb{N}^*$ and $x \in X$, strengthening \eqref{suprema}.

\noindent
Applying (ii), \eqref{suprema}  and (i) we see that 
$$
	\int |M_N(T)(f_{2j+2}-f_{2j+1})|^r d\mu =
	\int_{\bigcup_{k=0}^{N_{2j+2}}(T^{-k}E^{2j+2,2j+1})} |M_N(T)(f_{2j+2}-f_{2j+1})|^r d\mu \le
$$
$$
\frac1{N^r} 2^{r} \eps_{2j+2}^{-r}\Big(\frac{1}{\prod_{k=0}^{2j+1}\eps_{k}}\Big)^{r} \mu(\bigcup_{k=0}^{N_{2j+2}}(T^{-k}E^{2j+2,2j+1}))
\le \frac1{N^r} 2^{r+1} \eps_{2j+2}\eps_{2j+2}^{-r}\Big(\frac{1}{\prod_{k=0}^{2j+1}\eps_{k}}\Big)^{r}
$$
Hence, since $\eps_{2j+2}^2 \le \prod_{k=0}^{2j+1}\eps_k$, by \eqref{Lr}
$$
 \sum_{j=2J_0}^\infty \int |M_N(T)(f_{2j+2}-f_{2j+1})|^r d\mu \le
 \frac{2^{r+1}}{N^r} \sum_{j=2J_0}^{+\infty} \eps_{2j+2}^{1-3r} \le C\frac{2^{r+1}}{N^r}.
$$
Letting $N \to \infty$, \eqref{norms} yields 
$$\int |M_N(T)f-1|^r d\mu  \to 0.$$
This complete the proof of the first part of the proposition.
\medskip

The proof for $S$ is similar, as noted in \cite{Bu} (we have not listed all the properties of the $f_j$ obtained
in \cite{Bu}). For the limit of $M_n(S)f$ to be 0 and not 1, we define in this case
 $g_{J_0}=\sum_{j=0}^{2J_0-1} (f_{2j}-f_{2j+1})$, so now $\int g_{J_0}=0$.
\end{proof}
{\bf Remark.}  By \cite{Bu} we have also a.e. convergence $M_n(T)f \to 1$ and $M_n(S)f \to 0$.

\begin{prop} \label{example}
Let $T$ be induced by an irrational rotation of the circle, and let $0<r<\frac13$. Then
 ${\bf 1} \in \overline{(I-T)L^r}$.
\end{prop}
\begin{proof}
Let $T$ be induced by $\alpha$, and let $S$ be induced by some $\beta$ so that $\alpha$ and $\beta$ 
are independent over the rationals. We need the following lemma.



\begin{lem}\label{WCT}
Fix  $0<r<1$. Then    $\overline{(I-T)L^r}=\overline{(I-S)L^r}$.
\end{lem}
\begin{proof} By symmetry, it is enough to show $\overline{(I-T)L^r} \supset \overline{(I-S)L^r}.$
Obviously, we have 
	$$\overline{(I-T)L^r} \supset \overline{(I-T^n)L^r}.$$
By density of the orbits of $\alpha$, there exists a sequence $(n_j)$ of integers such that 
$\alpha^{n_j} \to \beta$. Hence for each $f \in L^1(X,\mu)$ we have $\|T^{n_j}f -Sf\|_1 \to 0$
(by approximation; easy for $f$ continuous). By H\"older's inequality 
$\int |f|^r d\mu \le (\int |f|d\mu)^r$,
so for $f \in L^1$ we obtain $\int|T^{n_j}f-Sf|^r d\mu \to 0$.
But $L^1(X,\mu)$ is dense in $L^r(X,\mu).$ Therefore, for each $f \in L^r$, 
we have  $\int|T^{n_j}f-Sf|^r d\mu \to 0$, which yields
$(I-T^{n_j})f \to  (I-S)f $ in $L^r$.
This implies that $(I-S)f \in \overline{(I-T)L^r}$ since, for each $j$, 
$(I-T^{n_j})(f) \in \overline{(I-T)L^r}$. This completes the proof of the Lemma.		
\end{proof}  

\noindent
{\it Proof of Proposition \ref{example}}. Let $ f$ be given by Proposition \ref{major}. 
Then $f \in L^r$, and by Theorem \ref{limit}, $f-{\bf 1} \in \overline{(I-T)L^r}$ and 
$f \in  \overline{(I-S)L^r}$. Whence, by Lemma \ref{WCT}, ${\bf 1} \in\overline{(I-T)L^r}$.
\end{proof}

By Proposition \ref{example}, for $0<r<1/3$ and any irrational rotation,
 ${\bf 1}\in\overline{(I-T)L^r}$, while its averages do not converge to zero (but do converge).
Kosloff and Voln\'y suggested the next general proposition (with an outline of its proof), 
in which the averages of a function in $\overline{(I-T)L^r}$, $0<r<1$, do not converge.  

\begin{prop} \label{kv}
Let $\theta$ be an ergodic aperiodic measure preserving tansformation
of a probability space $(X,\Sigma,\mu)$ and denote $Tf=f \circ \theta$. Then for $0<r<1$
there exists $f \in \overline{(I-T)L^r}$ such that $M_n(T)f$ does not converge in $L^r$ metric.
\end{prop}
\begin{proof}
Fix $0<r<1$ and fix $r<\alpha <1$. It is proved in \cite[Theorem 1]{KV} that there exists
a measurable (real) $f$ such that $\frac1{n^{1/\alpha}}\sum_{k=0}^{n-1}T^kf$
converges in distribution to a symmetric $\alpha$-stable distribution with parameter
$0< \sigma<1$. We show that $f$ is the desired function.

By formula (5.13) in \cite[Theorem 1, p. 576]{Fe} (see also \cite[Property 2.5, p. 63]{Zo}), when 
$g$ has $\alpha$-stable symmetric distribution and $r<\alpha$, we have $\int |g|^r d\mu <\infty$.  

{\it Claim: There exist $C$ and $K_0$, such that when $g$ has a symmetric $\alpha$-stable 
distribution with parameter $\sigma <1$, for every $r< \alpha$ and $K>K_0$,}
$$
\int_{\{|g| > K\}}|g|^r d\mu \le C \sigma^\alpha K^{r-\alpha}.
$$
For the proof, fix $g$ as in the claim, with distribution
$F(x)=\mu\big\{\omega :~~ g(\omega) \leq x\big\}, ~~x\in \mathbb{R},$
and following \cite{leo}, put $q(x)=1-F(x)+F(-x), ~~x\in \mathbb{R}.$
Then $q(x)=\mu\big\{\omega:~~ |g(\omega)|>x\big\}$ for $x \ge 0$.
By symmetry of $F$ we have $q(x)=2(1-F(x)).$
We further have, by \cite[p.367, B. 3]{leo},
$$\lim_{x \to +\infty}x^{\alpha} q(x) =\sigma^\alpha,$$
Whence,
$$\lim_{x \to +\infty}x^{\alpha} \mu\Big\{\omega ~~:~~ |g(\omega)| > x\Big\}=2 \sigma^\alpha.$$
Now, write
$$ \int_{\{|g| > K\}}|g|^r d\mu=\int_{\{|x| > K\}}|x|^r dF(x)=-2\int_{\{|x| > K\}}|x|^r dq(x),$$
and apply an integration by parts to get
$$  \int_{\{|g| > K\}}|g|^r d\mu=K^rq(K)+2r\int_{K}^{+\infty}x^{r-1}q(x)dx.$$
Let $K_0 >0$ such that for any $K>K_0$, we have
$$|K^\alpha q(K)-\sigma^\alpha|<\sigma^\alpha.$$
Then, for any $K>K_0$,
$$K^\alpha q(K) \leq 2 \sigma^\alpha,$$
and 
$$K^r q(K) \leq 2 \sigma^\alpha K^{r-\alpha}.$$
Therefore,
\begin{align}
\int_{\{|g| > K\}}|g|^r d\mu &\leq 2 \sigma^\alpha K^{r-\alpha}+2\sigma^\alpha r
\int_{K}^{+\infty}x^{r-1}x^{-\alpha} dx\\
&\leq 2 \sigma^\alpha K^{r-\alpha}+2\sigma^\alpha\frac{2r}{\alpha-r}  K^{r-\alpha}.
\end{align}
The claim thus follows by taking $C=\max\big\{2,\frac{2r}{\alpha-r}\big\}.$
\medskip

By the construction in \cite{KV}, $f$ is of the form $f=\sum_{j=1}^\infty(f_j-T^{d_j}f_j)$,
where $f_j=g_j 1_{\{2^j <g_j <2^{j^2}\}}$, with $g_j$ having a symmetric $\alpha$-stable 
distribution with parameter $\sigma = j^{-1/\alpha}$. Using the triangle inequality in $L^r$ 
and the claim, for $\log_2K_0 < n<m$ we obtain
$$
\int \Big| \sum_{j=n}^m(f_j-T^{d_j}f)\Big|^rd\mu \le
2\sum_{j=n}^m \int_{\{|g_j| >2^j\}}|g_j|^r d\mu \le 2 C\sum_{j=n}^m \frac1j 2^{j(r-\alpha)}.
$$
Since $\sum_{j=1}^\infty \frac1j 2^{j(r-\alpha)}<\infty$, the partial sums of $f$  are a Cauchy
 sequence in $L^r$ metric, so converge to $f$ in $L^r$, and since the partial sums are in $(I-T)L^r$
by their definition, $f$ is in $\overline{(I-T)L^r}$.

We denote below convergence in distribution by $\Rightarrow$. Let $\xi \not\equiv 0$ be a random
 variable on $(X',\mu')$ with the symmetric $\alpha$-stable distribution such that
$n^{-1/\alpha}\sum_{k=0}^{n-1}T^k \Rightarrow \xi$. Since $\xi$ has a continuous distribution, 
using \cite[Theorem 2.1]{Bi}, we obtain that
$$
\Big|n^{-1/\alpha}\sum_{k=0}^{n-1}T^kf\Big|^r \Rightarrow |\xi|^r.
$$
As noted above, $\int |\xi|^r d\mu'<\infty$.  
By  \cite[Theorem 3.4]{Bi} we have
$$
0< a:=\int |\xi|^r d\mu' \le \liminf_{n\to \infty} 
\int \big|n^{-1/\alpha}\sum_{k=0}^{n-1}T^kf\big|^rd\mu .
$$

\noindent Denoting $p=1/\alpha$, the last inequality yields that for  large $n$
$$\int \big|n^{-1}\sum_{k=0}^{n-1}T^kf\big|^rd\mu=
n^{r(p-1)}\int \big|n^{-1/\alpha}\sum_{k=0}^{n-1}T^kf\big|^rd\mu >\frac12a n^{r(p-1)}\to \infty.
$$
Hence $M_n(T)f$ does not converge in $L^r$.
\end{proof}

{\bf Remark.} {\it The averages $M_n(T)f$ of the above $f$ do not a.e. converge.} 
Indeed, if $M_n(T)f \to g$ a.e., then
$$
\frac1{n^{1/\alpha}}\sum_{k=0}^{n-1} T^kf =a
n^{1-\alpha^{-1}} M_n(T)f \to 0 \quad \text{a.e.}
$$
contradicting the convergence in distribution to $\xi \not\equiv 0$.

\begin{prop} \label{No-ae}
Let $T$ be induced by an ergodic aperiodic measure preserving transformation $\theta$
on a Lebesgue space $(X,\Sigma,\mu)$. Then for any $0<r<1$ there exists 
$f_r \in (I-T)L^r(\mu)$ such that $M_nf_r$ does not converge a.e.
\end{prop}
The proof will use the following proposition.
\begin{prop}\label{conze2}
Let $T$ be induced by an ergodic aperiodic measure preserving transformation $\theta$
on a Lebesgue space $(X,\Sigma,\mu)$. Then there exists $0\le h \in L^1(\mu)$ such that 
for every $0<r<1$ we have $\limsup_n n^{-r}T^nh(x)=\infty$ a.e.
\end{prop}
 \begin{proof} We first note that $\|n^{-r}T^ng\|_1 \to 0$ for $g \in L^1$ implies 
$\liminf |n^{-r}T^ng|=0$.

We use Rokhlin's lemma  for non-invertible transformations \cite{HS},
\cite{AC}\footnote{\cite{AC} does not require the additional assumption that $\theta$ is 
	onto, used in \cite{HS}.}
to construct the function $h$.  Let $\varepsilon_n=\e^{-n}$ and $k_n=2^n$. 
By Rokhlin's lemma there exists $B_n \in \Sigma$ such that $(\theta^{-j} B_n)_{0\le j <k_n}$
are disjoint and $k_n\mu(B_n)= \mu( \cup_{j=0}^{k_n-1}\,\theta^{-j}B_n) > 1-\varepsilon_n$.
	 Put $X_n:=\cup_{j=0}^{k_n-1}\,\theta^{-j}B_n$ and define
	$$
	h := \sum_{n=1}^\infty n^{-2} \frac1{\mu(B_n)} 1_{B_n}.
	$$ 
Then $\int h\,d\mu=\sum_{n=1}^\infty  n^{-2} < \infty$, so $0\le h \in L^1(\mu)$.

Let $x \in X_n \setminus B_n$. Then there exists $1\le j \le k_n-1$ such that 
$x \in \theta^{-j}B_n$.  For $0<r<1$ we then have
	 $$
\frac{T^jh(x)}{j^r} \ge \frac{h(\theta^j x)}{k_n^r} \ge 
	 \frac{n^{-2}}{\mu(B_n)}\cdot \frac1{k_n^r} = \frac{n^{-2}k_n^{1-r}}{k_n\mu(B_n)} 
	 \ge \frac{k_n^{1-r}}{n^2} =\frac{2^{n(1-r)}}{n^2}.
	 $$
We now fix $r\in(0,1)$. Given $M>0$, for $n>N_M$ we have $2^{n(1-r)}/n^2 >M$, so 
	  $$
	 \mu\Big(\Big\{x: \sup_{j\ge 1} \frac{T^jh(x)}{j^r} >M \Big\} \Big) \ge 
	 \mu(X_n \setminus B_n)=(k_n-1)\mu(B_n) >1-\varepsilon_n - k_n^{-1}   \qquad (n>n_M).
$$
Letting $n \to \infty$ we obtain $\sup_j j^{-r}T^jh >M$ a.e. Since this is for any $M>0$,  
	 we obtain $\limsup_{j\to \infty}j^{-r}T^jh= \infty$ a.e.
\end{proof}

{\it Proof of Proposition \ref{No-ae}.} Let $0\le h \in L^1$ the function obtained in 
Proposition \ref{conze2}.  For $0<r<1$ put $g_r=h^{1/r}$ and $f_r=(I-T)g_r$. Then 
$M_n(T) f_r=\frac1n(g_r-T^ng_r)$. But
$\limsup(T^ng_r/n)^r=\limsup T^nh/n^r =\infty$ a.e., so $M_n(T)f_r$ does not
 converge a.e. $\hfill \square$
 \medskip

{\bf Remark.} When $T$ is induced by a probability preserving transformation 
on $(X,\Sigma, \mu)$, Theorem \ref{main1} and Fatou's Lemma yield that if 
$f \in (I-T)L^r$, then $\liminf |M_n(T)f|=0$ a.e.
\smallskip

In Proposition \ref{conze2} we exhibit a positive function $h$ in  $L^1$ for which for any 
$r \in (0,1)$ we have $\sup_n \frac{T^nh}{n^r}=\infty$. In Proposition \ref{delB} below, 
we will establish that there is an abundance of such functions in $L^1$,  and in the more general setting,  for a given $r$ there are many functions in $L^p$, for any $1\le p <1/r$, with
$\sup_n \frac{|T^nh|}{n^r}=\infty$.  Note that when $g \in L^p$, $1<p<\infty$, we have 
$n^{-1/p}T^ng \to 0$ a.e. by the pointwise ergodic theorem.

We will use  the following special case of a theorem by del Junco and Rosenblatt \cite{delJB}. Its proof is inspired by Garsia's proof of Banach's principle \cite{G}. 

\begin{thm}\cite[Theorem 1.1]{delJB}\label{LJB}. Fix $1\le p\le \infty$ and
let $(T_n)$ be a sequence of linear maps on $L^p(X,\mu)$ which are continuous in measure. 
Assume that for every $\epsilon > 0$ and $K > 0$, there exists $h \in L^p(X,\mu)$
 with $\|h\| \le 1$ such that $\mu\Big\{x ~~:~~\sup_{n \geq 1}|T_nh|>K\Big\} >1-\epsilon$. 
Then there is a dense $G_\delta$ subset ${\Cal F} \subset L^p(X,\mu)$ such that
  $\sup_n |T_nh|=\infty$ a.e.   for any $h \in {\Cal F}$.
\end{thm}
 
\begin{prop}\label{delB} 
Let $T$ be induced by an ergodic aperiodic probability preserving transformation $\theta$
on a Lebesgue space $(X,\Sigma,\mu)$, and fix $ r \in (0,1)$. Then for any $1\le p < 1/r$ 
there is a dense $G_\delta$ subset ${\Cal F}\subset L^p(X,\mu)$, such that 
 $\limsup_n \frac{|T^nh|}{n^r}=\infty$ a.e.  for any $h \in {\Cal F}$.
\end{prop}
\begin{proof} We will apply Theorem \ref{LJB}. For $f \in L^p$, put $T_nf=\frac{T^nf}{n^r}$. 
Then $T_n$ is continous on $L^p$, and therefore continuous in measure.
Fix $\epsilon \in (0,1)$ and $K>0$, We use Rokhlin's lemma as in the proof of Proposition 
\ref{conze2}, and define $h$ as follows: 
$$
h := \sum_{n=1}^\infty n^{-2} \frac{1}{\mu(B_n)^{1/p}} 1_{B_n}.
$$ 
Then $0\le h \in L^p(\mu)$ with $\|h\|_p \le \sum_{n=1}^\infty  n^{-2}$. 
As in the proof of Proposition \ref{conze2}, there is $n_K \in \N$ such that for any $n >n_K$, 
we have
$$
\mu\Big\{x ~~:~~\sup_{n \geq 1}|T_nh|>K\Big\} >1-\epsilon.
$$
Let $h'=h/\|h\|_p$, and replace $K$ in the previous inequality by $K\|h\|_p$. Then $\|h'\|_p=1$, and
$$
\mu\Big\{x ~~:~~\sup_{n \geq 1}|T_nh'|>K\Big\} >1-\epsilon.
$$
We then deduce from Theorem \ref{LJB} that there is a dense $G_\delta$ subset
${\Cal F}\subset  L^p(X,\mu)$ such that $\sup_n \frac{|T^nh|}{n^r}=\infty$ a.e.  
for any $h \in {\Cal F}$.
\end{proof}
Proposition \ref{conze2} was inspired by its special case below, first proved by JP Conze; 
its "number theoretic" proof, which does not use Rokhlin's lemma, is of independent interest.

\begin{prop}\label{conze}
Let $T$ be induced by an irrational rotation of the unit circle. Then there 
exists $0\le h \in L^1$ such that for every $0<r<1$ we have $\limsup_n n^{-r}T^nh(x)=\infty$ a.e.
\end{prop}
\begin{proof}
We represent the transformation by $\theta x=x+\alpha \mod 1$, with $\alpha$ irrational. Let
$(q_n)$ be the sequence of denominators in the continued fractions representation of $\alpha$,
and define 
$$h:=\sum_{n=1}^\infty n^{-2}q_n 1_{[0,2/q_n]}.$$
 Then $\int_0^1 h\,dx =2\sum_{n \ge 1} n^{-2} <\infty$, so $h$ is integrable.

For $x \in[0,1)$ there exists $k_n(x)$, $0\le k_n(x)\le q_n$, such that 
$x+k_n(x)\alpha \mod 1 \in [0,2/q_n]$. The sequence $k_n(x)$ converges to $\infty$, and for $0<r<1$
we have
$$
\frac1{k_n(x)^r} h(\theta^{k_n(x)}x) \ge \frac1{q_n^r} h(\theta^{k_n(x)}x) \ge n^{-2}q_n^{1-r}
\to \infty,
$$
since $q_n$ increases to infinity at least exponentially \cite[Theorem 12, p. 13]{Kh}. 

\noindent Hence $\limsup_n  n^{-r}T^nh(x)=\infty$ a.e.
\end{proof}

\begin{prop} \label{rho-powers}
Let $T$ be induced by a probability preserving transformation on $(X,\Sigma, \mu)$ and fix $0<r<1$.
Given $\rho>1$, put $n_j:=[\rho^j]$. Then for $f \in (I-T)L^r(\mu)$ we have $\lim_j M_{n_j}f=0$
a.e.
\end{prop}
\begin{proof}
We show that if $g \in L^r$, then $n_j^{-1}T^{n_j}g \to 0$ a.e. Since $n_j+1 >\rho^j$, we have
$$
\int\Big|\frac{T^{n_j}g}{n_j}\Big|^r d\mu =
\Big(\frac{n_j+1}{n_j}\Big)^r \Big(\frac1{n_j+1}\Big)^r\int|g|^rd\mu \le 2^r (\rho^{-r})^j \int|g|^rd\mu.
$$
By the triangle inequality in $L^r$, we have
$$
\int \Big|\sum_{j=1}^\infty  \frac{T^{n_j}|g|}{n_j}\Big|^rd\mu \le
\sum_{j=1}^\infty \int \Big( \frac{T^{n_j}|g|}{n_j}\Big)^rd\mu \le 
2^r\int |g|^rd\mu \sum_{j=1}^\infty (\rho^{-r})^j < \infty.
$$
Hence the series $\displaystyle \sum_{j=1}^\infty  \frac{T^{n_j}|g|}{n_j}$ converges a.e., so 
 $ \displaystyle \frac{|T^{n_j}g|}{n_j}=\frac{T^{n_j}|g|}{n_j} \to 0$ a.e.

When $f =(I-T)g$ we have $\displaystyle  M_{n_j}f =n_j^{-1}(g-T^{n_j}g) \to 0$ a.e.
\end{proof}
\bigskip

\section{Ergodic theorems in the Hardy spaces $H^r(\mathbb T)$, $0<r<1$}

In this section,  the mean ergodic theorem for rotations is extended  to $H^r(\mathbb T)$, $0<r<1$. 
\smallskip
 
The Hardy space $H^r=H^r(\mathbb D)$, $r>0$, is the linear space of all holomorphic functions $f$ 
in the open unit disc $\mathbb D$ such that
$$\sup_{R <1}\frac{1}{2\pi}\int_{0}^{2\pi}|f(Re^{i t})|^r dt<+\infty.$$ 
Since the left hand-side in the above is increasing with $R$ \cite[Theorem 1.5]{Duren}, we have 
the limit
$$
\|f\|_{H^r}^r:=\lim_{R \to 1^-}\frac{1}{2\pi}\int_{0}^{2\pi}|f(Re^{i t})|^r dt<+\infty.
$$ 
When $f \in H^r$, it has radial limits $f(e^{it}):= \lim_{R \to 1^-} f(R e^{it})$ for a.e. 
$t \in[0,2\pi)$ (e.g. \cite[Theorem VII.7.25]{Zyg}). By Fatou's lemma 
$\frac1{2\pi}\int_0^{2\pi} |f(e^{it})|^r dt \le \|f\|_{H^r}^r$, so $f(e^{it}) \in L^r(\mathbb T)$. 
We denote by $H^r(\mathbb T)$ the set of all these radial limits.
In fact, $\frac1{2\pi}\int_0^{2\pi} |f(e^{it})|^r dt=\|f\|_{H^r}^r$ \cite[Theorem VII.7.32]{Zyg}, 
\cite[Theorem 2.6(i)]{Duren}, and where there is no risk of confusion we write $\|f\|_r^r$ for 
both sides of the latter equality.  Moreover \cite[Theorem 2.6]{Duren}, $f \in H^r$ satisfies
\begin{equation} \label{radial}
	\lim_{R \to 1^-} \frac1{2\pi} \int_0^{2\pi} \big| f(Re^{it}) - f(e^{it})\big| ^r dt = 0.
\end{equation}

For $R<1$, we denote by $\overline{\mathbb C_R}$ the circle of radius $R$.

For nice accounts on Hardy spaces, we refer to \cite[Chap. VII, Vol I, pp. 271-287]{Zyg},
\cite[Chap. III]{Katz}, \cite[Chap. 1]{Duren}. \\

It was observed in \cite[Chapter  VII, p. 284]{Zyg} that $H^r(\mathbb T)$ can be defined as 
the closure, in the $L^r$ metric, of the analytic trigonometric polynomials, and when 
$f(z)= \sum_{k=0}^{\infty} a_kz^k$ is in $H^r(\mathbb D)$, we have convergence of
 $\sum_{k=0}^n a_ke^{ikt}$ to $f(e^{it})$ in $L^r(\mathbb T)$ metric. Moreover, 
 $H^r(\mathbb T)$ is complete, so is a closed subspace of $L^r(\mathbb T)$.
 If $\theta$ is an irrational rotation of the circle with induced operator $Tf =f\circ \theta$,
then $e^{ikt}$ is an eigenfunction of $T$ for any $k \in \mathbb Z$, hence $H^r(\mathbb T)$
 is invariant under $T$.\\

In contrast to $L^r(\mathbb T)$, here we can take advantage of the theorem of Duren-Romberg-Shields 
(cf. \cite[Theorem 7.5, p.115]{Duren}), which describes the dual space of $H^r$. According to that 
theorem, the dual space of $H^r$ can be identified with a certain space of Lipshitz functions 
$\Lambda_\alpha$ for a certain $\alpha>0.$ \\

We prove the following results.
\begin{thm}\label{Key}Fix $0<r<1$ and let $T$ be induced by an ergodic irrational rotation 
$\theta$ of the circle. Then,
	$$H^r(\mathbb T) =\mathcal{C}\oplus\overline{(I-T)H^r(\mathbb T)},$$
	where $\mathcal{C}$ is the subspace of constant functions.
\end{thm}  
\noindent We further have the following ergodic theorem.

\begin{thm}\label{Hp} Fix $0<r<1$ and let $T$ be induced by an ergodic irrational rotation
$\theta$ of the circle. Then, for each $f \in H^r(\mathbb T)$ there is a constant 
$a(f) \in \mathbb{C}$ such that 
	$$\Big\|\frac1N\sum_{k=0}^{N-1}T^k f-a(f)\Big\|_r \to 0.$$
\end{thm}
Classically, the mean ergodic theorem for power-bounded operators in Banach spaces is equivalent to
 the ergodic decomposition. But, here, since the Hahn-Banach separation theorem fails, we need to
 give separate proofs for  Theorem \ref{Hp} and Theorem \ref{Key}. 
\begin{proof}[\textit{Proof of Theorem \ref{Key}.}] By Lemma \ref{limit} and the proof of Theorem \ref{main1}, we proceed first by proving that $\mathbf{1} \not \in \overline{(I-T)H^r}$. Assume, by contradiction, that $\mathbf{1} \in \overline{(I-T)H^r}$. Then, there is a sequence of functions $(g_n) \subset H^r$, such that
 $g_n- Tg_n  \to \mathbf{1}$  in $L^r$ metric.

By the Duren-Romberg-Shields's theorem,  elements $\psi$ in the dual of $H^r$ can be identified
with certain analytic Lipschitz functions $g$ by the relation 
	$$\psi(f)=\lim_{R \to 1}\frac{1}{2\pi}\int_{0}^{2\pi}f(Re^{it})g(e^{-it}) dt.$$
Since $g \equiv 1$ is in that class,  the linear form  $\phi$ on $H^r$ given by 
	$$\phi(f)=\lim_{R \to 1} \frac{1}{2\pi}\int_{0}^{2\pi}f(Re^{it}) dt,$$
	is an element of the dual space, i.e. well-defined and continuous \cite[p. 115]{Duren}. Therefore,
	$$\phi(g_n-Tg_n)\to \phi(\mathbf{1})=\mathbf{1}.$$
	But, for each $n \geq 1$,
	\begin{align*}
	\phi(g_n-Tg_n)&=\lim_{R \to 1} \Big(\frac{1}{2\pi}
	\int_{0}^{2\pi}	g_n(Re^{it})-g_n(Re^{i\theta(t)})dt\Big)\\
	&=\lim_{R \to 1}\Big( \frac{1}{2\pi}
	\int_{0}^{2\pi} g_n(Re^{it}) dt- \frac{1}{2\pi} \int_{0}^{2\pi} g_n(Re^{i\theta(t)})dt\Big)\\
	&=\lim_{R \to 1}\Big( \frac{1}{2\pi}
	\int_{0}^{2\pi} g_n(Re^{it}) dt- \frac{1}{2\pi} \int_{0}^{2\pi} g_n(Re^{it})dt\Big) =0,
	\end{align*} 
The last equality since $\theta$ is a measure preserving transformation.  Whence
 $$\phi(g_n-Tg_n) \to 0.$$  
This leads to a contradiction, so $\mathbf{1} \not \in \overline{(I-T)H^r(\mathbb T)}$. 
\smallskip

	Now, take $f \in H^r(\mathbb T)$ and write, in the $L^r$  metric,
	$$f(e^{it})=\sum_{k=0}^{+\infty}a_ke^{ikt}.$$
	For each $K \in \mathbb{N}^*$ we have $P_K(e^{it}):=\sum_{k=1}^{K}a_ke^{ikt} \in 
{(I-T)H^r(\mathbb T)}$, since for each $0\ne k \in \mathbb{Z}$, the function $t \mapsto e^{ikt}$ is an
 eigenfunction with unimodular eigenvalue $\ne 1$.
	Moreover, by von Neumann's ergodic theorem, we have
	$$\lim_{N \to +\infty}\frac{1}{2\pi}\int_{0}^{2\pi} \Big|\frac{1}{N}\sum_{n=1}^{N-1}P_K(e^{i\theta^n(t)})\Big|^r dt \leq 
	\lim_{N \to +\infty}\Big\|\frac{1}{N}\sum_{n=1}^{N-1}P_K(e^{i\theta^n(t)})\Big\|_2^r=0.$$
	We thus get $f-a_0=L^r- \lim_K P_K \in \overline{(I-T)H^r(\mathbb T)}.$ 
Hence $H^r(\mathbb T)=\mathcal{C}+\overline{(I-T)H^r(\mathbb T)}$; the sum is a direct sum, 
since we have proved that $\mathcal{C}\cap\overline{(I-T)H^r(\mathbb T)}=\{0\}$.  
\end{proof}

{\bf Remark.} In Proposition \ref{example} we proved for $T$ induced by an irrational rotation
 that $\mathbf{1} \in \overline{(I-T)L^r(\mathbb T)}$ when $0<r<\frac13$; 
the proof of Theorem \ref{Key}  shows $\mathbf{1} \not\in \overline{(I-T)H^r(\mathbb T)}$.
\medskip

For the the proof of Theorem \ref{Hp}, we need the following lemma; for its proof we refer to 
\cite[Chap. VII, Vol I pp. 284/5]{Zyg}.
\begin{lem}\label{MSO} 
Let $r \in (0,1)$ and $f \in H^r$ such that $\big\|f \big\|_r \leq M$, for some $M>0.$ Then,  
for $R<1$ and for 
	$|z| \leq R,$ we have
	$$\big|f(z)\big| \leq \frac{M}{\big(1-R\big)^{\frac{1}{r}}}.$$ 
\end{lem}
We proceed now to the proof of Theorem \ref{Hp}. 
\begin{proof}[\textit{Proof of Theorem \ref{Hp}.}] Let $f(z)=\sum_{k=0}^{+\infty}a_kz^k$ and $P_K(z)=\sum_{k=0}^{K}a_kz^k.$ Then, $\Big\|f-P_K\Big\|_r \to 0$ as $K \to \infty$. 
 Fix $\epsilon>0$ and $K_0$ such that for each $K >K_0$, we have
	$$  \Big\|f-P_K\Big\|_r  <\epsilon.$$
	Take $R<1$ sufficiently close to 1. Then, by Lemma \ref{MSO}, we have	
$$ 
\big|f(z)-P_K(z)\big|<\frac{\epsilon}{\big(1-R\big)^{\frac{1}{r}}} \quad \rm{for }\quad  |z|\leq R.
$$
	Therefore, the sequence $\Big(P_K(z)-f(z)\Big)_{K\ge 1}$ is uniformly bounded in each 
        compact subset of $|z| \le R<1$. We thus apply Montel-Stieltjes-Osgood theorem to 
	extract a subsequence $K_j$ such that $P_{K_j}(z)-f(z)$ converges uniformly to zero
	in each compact subset of $|z| \leq R<1$. Whence, there is 
$j_0 \in \mathbb{N}$ such that for each $j \geq j_0$, for any $t \in [0,2\pi)$ we have
	\begin{align*}
	\Big|P_{K_j}\big(Re^{i \theta^n(t)}\big)-f\big(Re^{i \theta^n(t)}\big)\Big|<\epsilon
	\end{align*}
	Hence, for any $N \geq 1,$ and for any $t \in [0,2\pi)$
	\begin{align}\label{eq:U1}
	\Big|\frac{1}{N}\sum_{n=0}^{N-1}P_{K_j}\big(Re^{i \theta^n(t)}\big)-
	\frac{1}{N}\sum_{n=0}^{N-1}f\big(Re^{i \theta^n(t)}\big)\Big|<\epsilon
	\end{align}
	But, by von Neumann's mean ergodic theorem, we have
	\begin{align}\label{erg:U2}
	\lim_{N \to +\infty}\Big\|\frac{1}{N}\sum_{n=0}^{N-1}P_{K_j}\big(e^{i \theta^n(t)}\big)-a_0\Big\|_2=0.
	\end{align}
	Therefore, there is $N_0 \in \mathbb{N}$, such that  for any $N >N_0$, we have
	\begin{align}\label{erg:U3}
	\int \Big|\frac{1}{N}\sum_{n=0}^{N-1}P_{K_j}\big(Re^{i \theta^n(t)}\big)-a_0\Big|^r dt&< 
	\int \Big|\frac{1}{N}\sum_{n=0}^{N-1}P_{K_j}\big(e^{i \theta^n(t)}\big)-a_0\Big|^r dt\\
	&< \Big\|\frac{1}{N}\sum_{n=0}^{N-1}P_{K_j}\big(e^{i \theta^n(t)}\big)-a_0\Big\|_2^r <\epsilon.
	\end{align}
	Combining \eqref{eq:U1} and \eqref{erg:U3} with  the triangle inequality, 
	we get, for any $N \geq N_0$, for each $R<1$,
	\begin{align}
	&	\int \Big|\frac{1}{N}\sum_{n=0}^{N-1}f\big(Re^{i \theta^n(t)}\big)-a_0\Big|^r dt \nonumber\\ 
	&\leq \int \Big|\frac{1}{N}\sum_{n=0}^{N-1}f\big(Re^{i \theta^n(t)}\big)-\frac{1}{N}\sum_{n=0}^{N-1}P_{K_j}\big(e^{i \theta^n(t)}\big)\Big|^r dt+   \int \Big| \frac{1}{N}\sum_{n=0}^{N-1}P_{K_j}\big(e^{i \theta^n(t)}\big)-a_0 \Big|^r dt \label{eq:U4}\\
	&\leq \epsilon+\epsilon=2\epsilon.\nonumber
	\end{align}
	Whence, for any $N \geq N_0$,
	$$\int \Big|\frac{1}{N}\sum_{n=0}^{N-1}f\big(Re^{i \theta^n(t)}\big)-a_0\Big|^r dt <2\epsilon.$$
	This finishes the proof of the theorem.
\end{proof}

The following corollary is in contrast to Propositions \ref{example} and \ref{kv}.
\begin{cor} \label{MET}
 Fix $0<r<1$. Let $T$ be induced by an ergodic irrational rotation $\theta$ of the circle, 
and let $f \in H^r(\mathbb T)$. Then $f \in \overline{(I-T)H^r(\mathbb T)}$ if and only if
$$
	\int_0^{2\pi} \Big|\frac1n\sum_{k=0}^{n-1}(T^k f)(e^{it})\Big|^rdt \to 0.
$$
\end{cor}
\begin{proof}
Let $f \in \overline{(I-T)H^r(\mathbb T)}$. Then the averages $M_nf$ are also in 
$\overline{(I-T)H^r(\mathbb T)}$. By the proof of Theorem \ref{Hp} $M_nf \to a_0$ in $L^r$ metric,
which yields $a_0  \in  \overline{(I-T)H^r(\mathbb T)}$, and by Theorem \ref{Key} $a_0=0$.

Assume now that $M_nf \to 0$ in $L^r$ norm. Since $H^r(\mathbb T)$ is invariant under $T$,  
the functions $g_n:=f-M_nf=\frac1n\sum_{k=1}^{n-1}(I-T)\sum_{j=0}^{k-1}T^jf$, 
defined in the proof of Theorem \ref{main1}, are in $(I-T)H^r(\mathbb T)$, and $\|f-g_n\|_r \to 0$.
\end{proof} 

\begin{thm} \label{sato2}
 Fix $0<r<1$. Let $T$ be induced on $H^r(\mathbb T)$ by an ergodic irrational rotation 
 of the circle, and let $f \in H^r(\mathbb T)$. Then $f \in (I-T)H^r(\mathbb T)$ if and only if 
\begin{equation} \label{gh}
\sup_{N\ge 1}\int_0^{2\pi} \Big|\frac1N \sum_{n=1}^N \sum_{k=0}^{n-1} T^kf(e^{it})\Big|^r dt < \infty.
\end{equation}
\end{thm}
\begin{proof}  
Assume \eqref{gh}. 
By the construction, $f(e^{it}) \in H^r(\mathbb T)$ is a.e. the boundary limit of $f(z) \in H^r$.
By \eqref{gh}, there exists $M>0$ such that 
	$\|\frac1N \sum_{n=1}^N\sum_{k=0}^{n-1} T^kf\|_r \le M$ 
for every $N\ge 1$.  By Lemma  \ref{MSO}, for every $N \ge 1$ and $R<1$  we have
\begin{align}
	\Big|\frac1N \sum_{n=1}^N \sum_{k=0}^{n-1} T^kf(z)\Big| \le \frac{M}{(1-R)^{1/r}} \qquad 
	\forall \quad |z|  \le R. \label{eq:U5}
\end{align}

Hence the sequence $\big\{\frac1N \sum_{n=1}^N \sum_{k=0}^{n-1} T^kf(z) \big\} \subset H^r$ is
 uniformly bounded on any compact subset of $\mathbb D$. By the Montel-Stieltjes-Osgood theorem, there
 is a subsequence $\{N_j\}$ such that $g_j(z):=\frac1{N_j} \sum_{n=1}^{N_j}\sum_{k=0}^{n-1} T^kf(z)$
 converges uniformly on compact subsets of $\mathbb D$, to a function $g$ continuous on $\mathbb D$.
 Since on any triangle in $\mathbb D$, $g$ is the uniform limit of analytic functions,
$g$ is analytic in $\mathbb D$ by Morera's theorem. Clearly $g_j \in H^r$, so for $R<1$, by 
Lebesgue's theorem and \eqref{gh},
$$
\int_0^{2\pi} |g(Re^{it})|^rdt = \lim_{j\to\infty} \int_0^{2\pi} |g_j(Re^{it})|^r dt \le
\liminf_{j\to\infty} \int_0^{2\pi} |g_j(e^{it})|^rdt \le M^r.
$$
Hence $g \in H^r$, with $\|g\|_r \le \liminf_j \|g_j\|_r \le M$, so $g(e^{it}) \in H^r(\mathbb T)$. 
We show that ${(I-T)g(e^{it})=f(e^{it})}$. 
\smallskip


Since $g_j \to g$  converges uniformly on compact subsets of $\mathbb D$, it follows that for $R<1$,
$$\int_0^{2\pi} |g_j(Re^{it})-g(Re^{it})|^r dt \to_{j\to\infty} 0.$$
This with continuity of $T$ yields, with limits in $L^r(\overline {\mathbb C_R})$ metric,
$$
(I-T)g (Re^{it}) =  \lim_j (I-T)g_j(Re^{it})= 
$$
$$
\lim_j \frac1{N_j}\sum_{n=1}^{N_j} (I-T^n)f(Re^{it}) =
 f (Re^{it}) - \lim_j M_{N_j}(T)f(Re^{it}) 
$$
Let $f(z)=\sum_{k=0}^{+\infty}a_k z^k$. Then,
by \eqref{eq:U4}, $M_n(T)f \to a_0$ in $L^r(\overline {\mathbb C_R})$ metric, so, for almost all $t$, 
$$[(I-T)g](Re^{it})=f(Re^{it})-a_0.$$
Taking radial limits we obtain $(I-T)g=f-a_0$ a.e. on $\mathbb T$. \\

{\it Claim:}  $a_0=0$.
\smallskip

Assume, by contradiction, that $a_0 \neq 0$. Let $R<1$. Then, for almost all $t$, 
$$(I-T)g(Re^{it})+a_0=f(Re^{it}).$$ 
We can thus write
\begin{align*}
\int_0^{2\pi} \Big|\frac1N \sum_{n=1}^N \sum_{k=0}^{n-1} T^kf(Re^{it})\Big|^r dt
&=\int_0^{2\pi} \Big|\frac1N \sum_{n=1}^N \Big( (I-T^n)g(Re^{it})-na_0\Big)\Big|^r dt\\
&=\int_0^{2\pi} \Big|g(Re^{it})-\frac1N \sum_{n=1}^N T^ng(Re^{it})-\frac{(N+1)}{2}a_0\Big|^r dt
\end{align*}
Applying the triangle inequality, we get
$$
\frac{1}{2\pi}\int_0^{2\pi} \Big|\frac1N \sum_{n=1}^N \sum_{k=0}^{n-1} T^kf(e^{it})\Big|^r dt \ge
\frac{1}{2\pi}\int_0^{2\pi} \Big|\frac1N \sum_{n=1}^N \sum_{k=0}^{n-1} T^kf(Re^{it})\Big|^r dt \ge
$$
$$
\frac{(N+1)^r}{2^r}|a_0|^r -
\frac{1}{2\pi}\int_0^{2\pi} \Big|g(Re^{it})-\frac1N\sum_{n=1}^N T^ng(Re^{it})\Big|^r dt\ge
	$$
	$$
	\frac{(N+1)^r}{2^r}|a_0|^r-\|g\|_r^r -\| M_N(T)g\|_r^r.
	$$
By Theorem \ref{Hp}, $\{\|M_N(T)g\|_r^r\}$ is bounded; hence
$$
\liminf_N\int_0^{2\pi} \Big|\frac1N \sum_{n=1}^N \sum_{k=0}^{n-1} T^kf(e^{it})\Big|^r dt=+\infty,
$$
which contradict \eqref{gh}. Hence $a_0=0$.  We thus obtain $(I-T)g=f$ in $H^r(\mathbb T)$. 
\medskip

Assume now that $f \in (I-T)H^r(\mathbb T)$, so $f=(I-T)g$. Then
$$
\frac1{2\pi} \int_0^{2\pi} \Big|\frac1N \sum_{n=1}^N \sum_{k=0}^{n-1} T^kf(e^{it})\Big|^r dt =
\frac1{2\pi}\int_0^{2\pi} \Big|\frac1N \sum_{n=1}^N \Big( (I-T^n)g(e^{it})\Big)\Big|^r dt =
$$
$$
\frac1{2\pi}\int_0^{2\pi} \Big|g(e^{it}) -M_N(T)g(e^{it})\Big|^r dt \le \|g\|_r^r+\|M_N(T)g\|_r^r.
$$
Since $M_N(T)g$ converges in $H^r(\mathbb T)$, the last term is bounded, so \eqref{gh} holds.
\end{proof}

{\bf Remark.} When $f \in (I-T)H^r(\mathbb T)$,  computations as in the end of the proof of Theorem
 \ref{sato2} yield
\begin{align}\label{Cob}
\sup_{n\ge 1} \int_0^{2\pi} \Big|\sum_{k=0}^{n-1}T^kf(e^{it})\Big|^r dt< \infty, 
\end{align}
If we assume \eqref{Cob}, then $f \in (I-T)L^r(\mathbb T)$ by Theorem \ref{sato}, but neither \cite{Sato}
nor \cite{AO} (which also applies) yield an indication of the "location" of the function $g \in L^r$ 
satisfying $(I-T)g=f$. The assumption \eqref{Cob} yields that $M_n(T)f \to 0$ in $H^r(\mathbb T)$, 
so $a_0=0$ by Theorem \ref{Hp}, but we have not succeeded to deduce from that assumption that 
$g \in H^r(\mathbb T)$, nor that it implies \eqref{gh}. 
\smallskip

\begin{prop} \label{Hr-ConzeII}
Let $T$ be induced by an irrational rotation of the circle. Then there exists $f$ such that
$f\in (I-T)H^r(\mathbb T)$  for each $0<r<1$, and $M_n(T)f$ does not converge a.e.
\end{prop}
\begin{proof} We produce a function $g$ with $g \in H^r(\mathbb T)$ for each $r<1$, such that 
	$\frac{T^ng}{n}$ does not converge a.e.

	Let $g(z)=\frac{1}{1-z}$; we show that $g(\e^{it})$ is in $H^r(\mathbb T)$, for each $r<1$. 
	We first prove $g \in L^r(\mathbb T)$:
	$$
\int_{0}^{2\pi} \Big|\frac{1}{1-e^{it}}\Big|^r dt =
	2^{-r}\int_{0}^{2\pi} \Big|\frac{1}{\sin(\frac{t}{2})}\Big|^rdt 
=2^{1-r} \int_{0}^{\pi/2} |\frac{1}{\sin(t)}\Big|^rdt.
$$
But for each $t \in [0,\pi/2)$, $\sin(t) \geq \frac{2}{\pi}t$. Therefore
$$
\int_{0}^{2\pi} \Big|\frac{1}{1-e^{it}}\Big|^r dt= 
2^{1-r} \sum_{n=1}^{+\infty}
\int_{\frac{\pi}{2(n+1)}}^{\frac{\pi}{2n}} \Big|\frac{1}{\sin(t)}\Big|^rdt \le
 2^{1-r} \sum_{n=1}^{+\infty}
\int_{\frac{\pi}{2(n+1)}}^{\frac{\pi}{2n}}\frac{\pi^r}{2^r}\frac{1}{t^r}dt \leq 
$$
$$
2^{1-2r} \pi^r \sum_{n=1}^{+\infty} 
\frac{1}{1-r}\Big(\Big(\frac{\pi}{2n}\Big)^{1-r}-\Big(\frac{\pi}{2(n+1)}\Big)^{1-r}\Big) \le
$$
$$
\lim_{N \to \infty} 2^{1-2r} \pi^r 
\frac{1}{1-r}\Big(\Big(\frac{\pi}{2}\Big)^{1-r}-\Big(\frac{\pi}{2(N+1)}\Big)^{1-r}\Big)
\leq \frac{2^{1-3r}\pi}{1-r}.
$$
	Thus $g \in L^r(\mathbb T)$. We now show $g$ is in $H^r(\mathbb T)$.
Let $P_n(\e^{it})=\frac1N\sum_{n=1}^N \sum_{k=0}^{n-1} \e^{ikt}$. Then
	for $0<t<2\pi$ we have 
	$$
	P_N(\e^{it})= (1-\frac1N\sum_{n=1}^N \e^{int})/(1-\e^{it}).
	$$ 
	Hence $P_N(\e^{it}) \to g(\e^{it})$ for $0<t<2\pi$, and 
$g(\e^{it})- P_N(\e^{it})= (\frac1N\sum_{n=1}^N \e^{int})/(1-\e^{it}).$ Thus,
 almost everywhere $|g(\e^{it})-P_N(\e^{it}|^r \le \frac1{|1-\e^{it}|^r} \in L^1(\mathbb T)$.
	By Lebesgue's dominated convergence theorem, 
	$\int_0^{2\pi}|g(\e^{it})-P_N(\e^{it}|^r dt \to 0$. Thus $g$ is approximated in $L^r$
	distance by analytic trigonometric polynomials, so $g \in H^r(\mathbb T)$,
\smallskip
	
As in Proposition \ref{conze}, we represent the transformation by $\theta x=x+\alpha \mod 1$,
	with $\alpha$ irrational. Let $(q_n)$ be the denominators in the continued fractions
	representation of $\alpha$. As noted in the proof of Proposition \ref{conze}, for 
$x \in [0,1)$ there is $\ell_n(x) \in [0,q_n]$ such that $x+\ell_n(x)\alpha \mod 1\in [0,2/q_n]$.
	Then, for any $x \in [0,1)$ we have
$$\Big|1-e^{2\pi i(\ell_n(x)\cdot \alpha+x)}\Big| \leq \frac{2}{q_n},$$
and we obtain
$$
	\frac{(T^{\ell_n(x)}g)(\e^{2\pi ix})}{\ell_n(x)} =
\frac1{\ell_n(x)\Big|1-e^{i(\ell_n(x).\alpha+x)}\Big|} \geq \frac{q_n}{2\ell_n(x)} \geq \frac12,
 $$
	so $\limsup \Big|\frac{T^ng}{n}\Big|\geq \frac12$ a.e. Since 
$\int_0^{2\pi} \Big|\frac{T^ng}n\Big|^r dt= \int_0^{2\pi} \Big|\frac{g}n\Big|^r dt \to 0$,
$\liminf \Big|\frac{T^ng}{n}\Big|=0$ a.e. by Fatou's lemma. We conclude that $\frac{T^ng}{n}$ 
does not converge a.e. To finish the proof, we put $f=(I-T)g$ and the proof is complete.
\end{proof}

{\bf Remarks.} 1. By Corollary \ref{MET} and Fatou's lemma, every
$f \in \overline{(I-T)H^r(\mathbb T)}$ satisfies $\liminf |M_nf| =0$ a.e.

2. Note that Proposition \ref{rho-powers} applies to any $f \in (I-T)H^r(\mathbb T)$.

3. Here is an alternative proof that $1/(1-\e^{it}) \in H^r(\mathbb T)$ for any $0<r<1$.
We prove that $g(z)=\frac{1}{1-z}$ is in $H^r(\mathbb D)$, for each $r<1$. Indeed, for $0<R<1$, write
$$
\int_{0}^{2\pi} \Big|\frac{1}{1-Re^{it}}\Big|^r dt=
\int_{0}^{2\pi} \Big(\Big|\frac{1}{1-Re^{it}}\Big|^2\Big)^{r/2} dt=
\int_{-\pi}^{\pi} \Big(\frac{1}{1+2R\cos(s)+R^2}\Big)^{r/2}ds .
$$
The last term is obtained by change of variable $s= t-\pi$.  Now, put $ u= \tan(s/2)$ to obtain
$$
\int_{0}^{2\pi} \Big|\frac{1}{1-Re^{it}}\Big|^r dt=\frac12 \int_{-\infty}^{+\infty}
\Big(\frac{1}{1+2R\frac{1-u^2}{1+u^2}+R^2}\Big)^{r/2} \frac{du}{1+u^2}.
$$
But for each $u \in \R$, $1+2R\frac{1-u^2}{1+u^2}+R^2=\Big(R+\frac{1-u^2}{1+u^2}\Big)^2+1-\Big(\frac{1-u^2}{1+u^2}\Big)^2 \geq \frac{4u^2}{(1+u^2)^2}$. Therefore
$$
\sup_{0<R<1}\int_{0}^{2\pi} \Big|\frac{1}{1-Re^{it}}\Big|^r dt \leq
\frac12 \int_{-\infty}^{\infty} \Big(\frac{(1+u^2)^2}{4u^2}\Big)^{r/2} \frac{du}{1+u^2}=
$$
$$
\frac1{2^{r+1}} \int_{-\infty}^{\infty} \frac{1}{|u|^r} \frac{du}{(1+u^2)^{1-r}}<+\infty.
$$
Hence $g\in H^r(\mathbb D)$, for each $r<1$, so $g(\e^{it}) \in H^r(\mathbb T)$. 
\smallskip

4. Let $T=T_\alpha$ be induced by an irrational rotation of the circle, and $g(z)=1/(1-z)$
for $z \in \mathbb D$. We showed above that $g(e^{it})\in H^r(\mathbb T)$. 
Since $g(z)=\sum_{n=0}^\infty z^n$ on $\mathbb D$, by Theorem \ref{Hp} 
$\int_0^{2\pi} \Big|\frac1n\sum_{k=0}^{n-1}T^k_\alpha g(\e^{it})-1\Big|^rdt \to 0$.
Motivated by \cite{SU}, for $\varepsilon >0$ we define 
$$
E_n:=\{(\alpha,x): \Big|\frac1n\sum_{k=0}^{n-1}T^k_\alpha g(\e^{2\pi ix})-1\Big| > \varepsilon\}
\quad \text{and} \quad 
E_{n,\alpha}:= \{x: \Big|\frac1n\sum_{k=0}^{n-1}T^k_\alpha g(\e^{2\pi ix})-1\Big|>\varepsilon\}.
$$
We denote by $\lambda$ Lebesgue's measure on $[0,1)$. By Markov's inequality,  for each irrational
 $\alpha$ in $[0,1]$ we obtain $\lambda(E_{n,\alpha})\to 0$ as $n \to \infty$. 
By Fubini's theorem and Lebesgue's theorem,
$$
\lambda\times \lambda(E_n)= \int_0^1 \lambda(E_{n,\alpha}) d\lambda(\alpha) \to 0 \quad 
\text{as} \quad  n\to \infty,
$$
which means that $ \displaystyle  \frac1n\sum_{k=0}^{n-1}\frac1{1-\e^{2\pi i(k\alpha+x)}}$ converges  to 1
in $\lambda\times \lambda$ measure. Then there exists a subsequence $(n_j)$ such that 
$\displaystyle  \frac1{n_j}\sum_{k=0}^{n_j-1}\frac1{1-\e^{2\pi i(k\alpha+x)}}$ converges  to 1 
$\ \lambda\times\lambda$ a.e. Hence the limiting distribution in  the Sinai-Ulcigrai
Theorem \cite[Theorem 1]{SU} is $\delta_1$.
\medskip

The proof of the next proposition shows that for $r <\frac12$ there are more functions
$f$ satisfying the property  in Proposition \ref{Hr-ConzeII}.

\begin{prop}  \label{Hr-Conze}
Let $T$ be induced by an irrational rotation of the circle. Then there exists $f$ such that
$f \in (I-T)H^r(\mathbb T)$ for any $0<r<\frac12$, and $M_n(T)f$ does not converge a.e.
\end{prop}
\begin{proof} Let $0\le g\in L^1(\mathbb T)$ such that  $\limsup_n n^{-r}T^ng=\infty$ a.e.
 for any $r<1$, e.g. the function obtained in Proposition  \ref{conze}; we may assume 
(by adding 1) that $g >0$ everywhere. Let
	$$
	g(R\e^{it}):= \sum_{n=-\infty}^\infty R^{|n|}\hat g(n)\e^{int}, \qquad (0\le R<1)
	$$
where $\hat{g}(n)$ is the Fourier coefficient of $g$ at $n$,
	be the extension of $g$ to the unit disc, and denote by $\tilde{g}$ 
its harmonic conjugate, that is, 
$$
\tilde{g}(Re^{i t})=-i \sum_{n=-\infty}^{+\infty}\textrm{sgn}(n)R^{|n|} \hat{g}(n) e^{int},
	\qquad (0\le R<1)
$$ 
where $\textrm{sgn}(n)$ is the sign function. Then     
$\tilde g(\e^{it})= \lim_{R\to 1}\tilde{g}(Re^{i t})$ is defined a.e. \cite[Lemma 1.3, 
p. 64]{Katz}. 
	The function $G(z):=g(R\e^{it})+i\tilde g(R\e^{it})$ is analytic on $\mathbb D$,
	and  $\lim_{R\to 1}G(R\e^{it})=h(\e^{it}):=g(\e^{it})+i \tilde{g}(\e^{it})$ exists a.e.
	Since $g \in L^1(\mathbb T)$, by  \cite[Corollary 1.6, p. 66]{Katz} (see the 
	proof of \cite[Corollary 3.9, p. 87]{Katz}), we have 
$h(\e^{it}):=g(\e^{it})+i \tilde{g}(\e^{it}) \in H^\beta(\mathbb \T)$ for any $\beta<1$.

	Fix $0<r<\frac12$.  Put $h_r(\e^{it}):=h^2(\e^{it}) = \lim_{R\to 1}[G(R\e^{it})]^2$.
	Since $G^2$ is analytic, and $G \in H^{2r}(\mathbb D)$, 
	we have $G^2 \in H^r(\mathbb D)$.  Then $\lim_{R\to1}G^2(R\e^{it})=h^2(\e^{it})$ is in
	$H^r(\mathbb T)$.
Since $\Re (h)\ge 1$, $\log   h$ is well-defined, and $h= \e^{\frac12 2\Log h}=(h^2)^{1/2}$.
But $\frac{T^nh_r}{n}$ does not converge a.e., since
$\big(\frac{T^nh_r}{n}\big)^{1/2}=\frac{T^ng+i T^n\tilde{g}}{n^{1/2}}$ does not converge. 
Consequently, when $f=(I-T)h^2$, $M_n(T)f$ does not converge a.e.
\end{proof}
\section{Appendix  by Guy Cohen: on rates of convergence}

Proposition \ref{conze2} raises the question of the rate of a.e. convergence to zero of 
$\frac1n T^ng$ when $g \in L^1(\mu)$. When $g \in L^p$, $1<p<\infty$, we have 
$n^{-1/p}T^ng \to 0$ a.e., so we have a rate of $1/n^{1-1/p}$. 
Proposition \ref{conze2} yields $h \in L^1$ which does not have this rate for any $p>1$.
The following proposition extends Proposition  \ref{conze2}, with  similar use of Rokhlin's lemma.

\begin{prop} \label{guy}
Let $T$ be induced by an ergodic aperiodic measure preserving transformation $\theta$
on a Lebesgue space $(X,\Sigma,\mu)$, and let $(b_n)$ be a sequence tending to infinity, 
such that $b_n/n$ decreases to zero. Then there exists $0\le h \in L^1(\mu)$ such that 
$\limsup_{n\to\infty} b_n\cdot n^{-1}T^nh=\infty$ a.e.
\end{prop}
\begin{proof}
Put $a_n=n/b_n$. Then $a_n$ increases to infinity and $a_n/n$ tends to zero.
For a fast increasing sequence $(\ell_n) \subset \mathbb N$ (to be determined later),
let $k_n $ be the first index such that $a_{k_n} \ge \ell_n$, so $k_n$ is increasing,
and $k_n^{-1}\ell_n \le k_n^{-1} a_{k_n} \to 0$.

Let $\varepsilon_n$ decrease to zero. By Rokhlin's lemma there is $B_n \in \Sigma$ such that
$(\theta^{-j}B_n)_{0\le j < k_n}$ are disjoint and $\mu(X_n) >1-\varepsilon_n$, where
$X_n=\cup_{j=0}^{k_n-1} \theta^{-j}B_n $. We define 
$h=\sum_{n=1}^\infty n^{-2}\frac1{\mu(B_n)}1_{B_n}$, so $0 \le h \in L^1$.
For $x \in X_n\setminus B_n$ there exists $1\le j \le k_n-1$ such that 
$x \in \theta^{-j}B_n$, and then
$$
\frac{T^jh(x)}{a_j} \ge \frac{h(\theta^jx)}{a_{k_n-1}} \ge 
\frac{n^{-2}}{\mu(B_n)}\cdot\frac1{a_{k_n-1}} \ge \frac{n^{-2}}{\mu(B_n)}\cdot\frac1{\ell_n} 
\ge \frac{k_n}{n^{2}\ell_n}.
$$
The sequence $(a_n)$ is given (through $(b_n)$ ), and the sequence $(\ell_n)$ can be chosen 
to grow fast enough so $\frac{k_n}{n^{2}\ell_n} \to \infty$ (see below). Then for $M>0$ we have
$\frac{k_n}{n^{2}\ell_n}>M$ for $n>N_M$.  We then have 
$$
b_j\cdot \frac{T^jh(x)}j=\frac{T^jh(x)}{a_j} > M \qquad
\qquad (n>N_M, \ x \in X_n\setminus B_n).
$$
Hence 
 $$
\mu\Big(\Big\{x: \sup_{j\ge 1} b_j\cdot \frac{T^jh(x)}{j} >M \Big\} \Big) \ge 
 \mu(X_n \setminus B_n)=(k_n-1)\mu(B_n) >1-\varepsilon_n - k_n^{-1}   \qquad (n>n_M).
$$
Letting $n \to \infty$ we obtain $\sup_j b_j\cdot j^{-1}T^jh >M$ a.e. Since this is for any $M>0$,  
	we obtain $\limsup_{j\to\infty} b_j\cdot j^{-1}T^jh =\infty$ a.e.

We now show how to choose $(\ell_n)$. We define an "inverse" of $(a_n)$ by 
$c_n= \min\{m: a_m\ge n\}$. It is easy to show that $c_n$ is increasing to infinity. 
When $c_n=m$, we have $a_{m-1}<n \le a_m$, so 
$$
\frac{c_n}n =\frac{m}n \ge \frac{m}{a_m} \to \infty.
$$ 
Fix $n$; then for $j >J_n$   we have $j^{-1}c_j> n^3$. Put $\ell_n=J_n+1$. By definition
$k_n=c_{\ell_n}$, so we have
$$
\frac{k_n}{n^2 \ell_n}=\frac{c_{\ell_n}}{n^2\ell_n}>n,
$$
hence $\frac{k_n}{n^2 \ell_n} \to \infty$.
\end{proof}

{\bf Remarks.} 1. The function $h$ depends in general on $(b_n)$, since $(\ell_n)$ and $(k_n)$ do.

2. In Proposition \ref{conze2}, obtained when $b_n=n^{1-r}$, $0<r<1$, 
 $h$ does not depend on $r$.

3. Proposition \ref{guy} does not apply to the sequence $b_n=n$. However,
for any unbouded $g \in L^1$, infinite returns of $\theta^nx$ to the set $\{y: |g(y)|>M\}$
show that $\limsup_n |T^ng| \ge M$ a.e. Hence $\limsup_n |T^ng| =\infty$ a.e.
(this is the idea in \cite[p. 391]{Ha} for proving $\limsup_n |T^ng| =\|g\|_\infty$ for $g \in L_\infty$).

\begin{cor} \label{no-rate}
Let $T$ be induced by an ergodic aperiodic measure preserving transformation $\theta$
on a Lebesgue space $(X,\Sigma,\mu)$, and let $0<\delta_n \to 0$. If there exists $(\gamma_n)$
such that $\gamma_n \to 0$, $\ n\gamma_n$ increase to $\infty$, and 
$\delta_n \le \gamma_n$ for every $n$, then there exists an $L^1$-coboundary $f \in (I-T)L^1$ 
such that $\limsup_{n\to\infty} \delta_n^{-1}|M_nf |\to \infty$ a.e.
\end{cor}
\begin{proof}
Let $b_n=1/\gamma_n$; then $(b_n)$ satisfies the assumptions of Proposition \ref{guy},
and let $h$ be the function obtained in Proposition \ref{guy}. Then $f=(I-T)h$ satisfies 
$$
\limsup_{n\to\infty} \delta_n^{-1}|M_nf | \ge \limsup_{n\to\infty} \gamma_n^{-1}|M_nf | =
\limsup_n b_n n^{-1}|h-T^nh| =\infty \quad a.e.
$$
\end{proof}
{\bf Remark.} Krengel \cite{Kr} and Kakutani-Petersen \cite{KP} proved that for any 
$0<\delta_n \to 0$ (with no additional conditions) there exists a bounded $f$ 
such that $\limsup_{n\to\infty} \delta_n^{-1}|M_nf |\to \infty$ a.e. The novelty in Corollary
\ref{no-rate} is in obtaining $f$ a coboundary.

\begin{prop}\label{delB1} 
Let $T$ be induced by an ergodic aperiodic probability preserving transformation $\theta$
on a Lebesgue space $(X,\Sigma,\mu)$. and let $(b_n)$ be a sequence tending to infinity, 
such that $b_n/n$ decreases to zero. Then there exists a dense $G_\delta$ subset 
${\Cal F}\subset L^1(X,\mu)$, such that 
 $\limsup_n b_n\cdot n^{-1}|T^nh| =\infty$ a.e.  for any $h \in {\Cal F}$.
\end{prop}
The proof is similar to that of Proposition \ref{delB}, using Proposition \ref{guy}.
\bigskip

{\bf Acknowledgement.} The first author is grateful to the Center for Advanced Studies in 
Mathematics of Ben-Gurion University for its hospitality and support. The second author 
is grateful to the University of Rouen-Normandie for its hospitality and support.

\end{document}